\documentclass[12pt,leqno]{article}

\usepackage{latexsym,amsmath,amssymb,amsthm,amsfonts}
\usepackage[active]{srcltx}     
\usepackage[all]{xy}

\renewcommand{\epsilon}{\varepsilon}

\newcommand{\R}{\mathbb{R}}
\newcommand{\N}{\mathbb{N}}

\newcommand{\p}{\mathcal{P}}
\newcommand{\dist}{\mathrm{dist}}
\newcommand{\x}{{\bf x}}

\newtheorem{theorem}{Theorem}
\newtheorem{lemma}{Lemma}
\newtheorem{proposition}{Proposition}

\begin{document}

\author{Andriy Bondarenko\footnote{This work was carried out during the tenure of an ERCIM ``Alain Bensoussan'' Fellowship Programme. The research leading to these results has received funding from the European Union Seventh Framework Programme (FP7/2007-2013) under grant agreement $n\textsuperscript{o}$ 246016.} ,
Danylo Radchenko, and Maryna Viazovska}
\title{Well-separated spherical designs}
\date{}
\maketitle
\begin{abstract}
For each $N\ge C_dt^d$ we prove the existence of a well-separated
spherical $t$-design in the sphere $S^d$ consisting of $N$ points,
where $C_d$ is a constant depending only on $d$.
\end{abstract}

\section{Introduction}
In this paper we will discuss the interrelation between several
classical optimization problems on spheres $S^d$ such as minimal
equal-weight quadratures (spherical designs), best packing problems,
and minimal energy problems. For $d=1$, a regular polygon is an optimal
configuration for all of these problems. However, for $d\ge 2$ exact solutions are known in very few cases.
Even asymptotically optimal configurations are sometimes very hard to
obtain (see for example Smale's 7th Problem~\cite{SS}).

We will prove the existence of certain configurations in $S^d$
which are spherical $t$-designs with asymptotically minimal number
of points and that simultaneously have asymptotically the best
separation property. These configurations also provide approximate
solutions for several other optimization problems.

Let $S^d=\left\{x\in\R^{d+1}:|x|=1\right\}$ be the unit sphere in
$\R^{d+1}$ equipped with the Lebesgue measure $\mu_d$ normalized by
$\mu_d(S^d)=1$. A set of points $x_1,\ldots,x_N\in S^d$ is called a
{\it spherical $t$-design} if
$$ \int_{S^d}P(x)\,d\mu_d(x)=\frac 1N\sum_{i=1}^N P(x_i) $$ for all
 polynomials in $d+1$ variables, of total degree at most
$t$. The concept of a spherical design was introduced by Delsarte,
Goethals, and Seidel~\cite{DGS}. For each $d, t\in {\mathbb N}$
denote by $N(d,t)$ the minimal number of points in a spherical
$t$-design in $S^d$. The following lower bound
\begin{equation}
\label{hh} N(d,t)\ge \begin{cases}\displaystyle
{{d+k}\choose{d}}+{{d+k-1}\choose{d}}&\text{if $ t=2k$,}\\&\\
\displaystyle   2\,{{d+k}\choose{d}}& \text{if $
t=2k+1$,}\end{cases}
\end{equation}
is proved in~\cite{DGS} (see also the classical monograph~\cite{CS}).
On the other hand, it follows from the general result by Seymour and Zaslavsky
\cite{SZ} that spherical designs exist for all positive
integers $d$ and $t$. The method of proof used in~\cite{SZ} was not constructive and authors
did't indicate an upper bound for $N(d,t)$ in terms of $d$ and $t$.
First feasible upper bounds were given by
Wagner~\cite{Wag} ($N(d,t)\le C_dt^{Cd^4}$) and Bajnok~\cite{B} ($N(d,t)\le
C_dt^{Cd^3}$). Korevaar and Meyers \cite{KM} have
improved these inequalities by showing that $N(d,t)\le
C_dt^{(d^2+d)/2}$. They have also conjectured that $N(d,t)\le
C_dt^d$. Note that \eqref{hh} implies $N(d,t)\ge c_dt^d$. Here and in what
follows we use the notations $C_d$, $L_d$, etc. ($c_d$,
$\lambda_d$, etc.) for sufficiently large (small) constants
depending only on $d$.

Korevaar and Meyers were motivated by the following problem coming
from potential theory: How to choose $N$ equally charged  points \rm
$x_1,\ldots,x_N$ in $S^2$ to minimize the value
$$
U_r(x_1,\ldots,x_N):=\sup_{|x|=r}\left|\frac 1N\sum_{i=1}^N\frac
1{|x-x_i|}-1\right|,\quad r\in (0,1)?
$$
The classical Faraday cage phenomenon states that any stable charge
distribution on the compact closed surface cancels the electric
field inside the surface. According to this model the minimal value
of $U_r$ should rapidly decay to 0, when $N$ grows.

It was shown in~\cite{KM} that if the set of points $x_1,\ldots,x_N$ is
a spherical $t$-design for some $t>cN^{1/2}$ then
$$
U_r(x_1,\ldots,x_N)\le r^{\alpha N^{1/2}}.
$$
The estimate is optimal up to the constant in the power.

Recently we have suggested a nonconstructive
approach to obtain an optimal asymptotic bound for $N(d,t)$ based on
the application of the topological degree theory; see~\cite{BRV,BV1}. We have proved the
following

\noindent {\sc {\bf Theorem A.}}\emph{ For each $N\ge C_dt^d$ there
exists a spherical $t$-design in $S^d$ consisting of $N$ points.}\\
\noindent This implies the Korevaar-Meyers conjecture.

Now we will give the definition of a well-separated sequence of
configurations. A sequence of $N$-point configurations
$X_N=\{x_{1N},\ldots,x_{NN}\}$ in $S^d$ is called well-separated if
\begin{equation}
\label{sep0}\min_{1\le i<j\le N}|x_{iN}-x_{jN}|\ge \lambda_dN^{-1/d}
\end{equation}
for some constant $\lambda_d$ and all $N\ge 2$.  The
inequality~\eqref{sep0} is optimal up to the constant $\lambda_d$.
That is, there exists a constant $L_d$ such that for any $N$-point
configuration $\{x_{1},\ldots,x_{N}\}$
$$
\min _{1\le i<j\le N}|x_i-x_j|<L_dN^{-1/d}.
$$
Many authors have predicted the existence
of well-separated spherical $t$-designs in $S^d$ of asymptotically
minimal cardinality $O(t^d)$ as $t\to\infty$ (see, e.g.~\cite{ACSW}
and~\cite{HL}). Moreover, in~\cite{HL} it was shown that if such
spherical designs exist then they have asymptotically minimal Riesz
$s$-energy. In this paper we prove the existence of above mentioned
spherical designs. Our main result is:
\begin{theorem}\label{well sep}
For each $d\ge 2$ there exist positive constants $C_d$ and
$\lambda_d$ depending only on $d$ such that for each $t\in\N$ and each  $N>C_dt^d$ there exists a spherical $t$-design in $S^d$
consisting of $N$ points $\{x_i\}_{i=1}^N$ with $|x_i-x_j|\geq \lambda_d
N^{-1/d}$ for $i\neq j$.
\end{theorem}

Theorem~\ref{well sep} is a natural generalization of Theorem A. The
paper is organized as follows. In Section~\ref{idea} we will reduce
Theorem~\ref{well sep} to the construction of a certain $N$-tuple of
maps $\x_1,\ldots,\x_N:\p_t\to S^d$. Then in Section~3 we will prove
several auxiliary results concerning area regular partitions of
sphere $S^d$ and spherical Marcinkiewicz-Zygmund type inequalities.
Finally in Section~4 we will construct the maps $\x_1,\ldots,\x_N$,
proving Theorem~1.

\section{Application of topological degree theory}\label{idea}
We will use the approach similar to that of~\cite{BRV}.
Let $\mathcal{P}_t$ be the Hilbert space of polynomials $P$ of
degree at most $t$ on $S^{d}$ such that $$
\int_{S^d}P(x)d\mu_d(x)=0,
$$
equipped with the usual inner product
$$
(P,Q)=\int_{S^d}P(x)Q(x)d\mu_d(x).
$$
By Riesz representation theorem, for each point $x\in S^d$ there
exists a unique polynomial $G_x\in \p_t$ such that
$$(G_x,Q)=Q(x) \;\;\mbox{for all}\;\;Q\in\p_t.$$
Then a set of points $x_1,\ldots,x_N\in S^d$ forms a spherical
$t$-design if and only if
\begin{equation}
\label{sph} G_{x_1}+\cdots+G_{x_N}=0.
\end{equation}
The gradient of a  differentiable function $f: \R^{d+1}\to \R$ is denoted by $$ \frac{\partial f}{\partial
x}:=\left(\frac{\partial f}{\partial \xi_1},\ldots,
\frac{\partial f}{\partial \xi_{d+1}}\right),\qquad x=(\xi_1,\ldots,\xi_{d+1}).$$

\noindent For a polynomial $Q\in\p_t$ we define the \emph{spherical
gradient}
$$\nabla Q(x):=\frac{\partial}{\partial x}\bigg( Q\Big(\frac {x}{|x|}\Big)\bigg).$$

\noindent Now we will use the following result from topological degree theory~\cite[Ths. 1.2.6 and 1.2.9]{OCC}.\\
{\sc {\bf Theorem B.}} {\it  Let $f: \R^n\to \R^n$ be a continuous map and
$\Omega$ an open bounded subset, with boundary $\partial\Omega$,
such that  $0\in\Omega\subset \R^n$. If $(x, f(x))> 0$ for all
$x\in\partial \Omega$, then there exists $x\in \Omega$ satisfying $f(x)=0$.}\\
We will apply Theorem B to the following open subset of a vector
space $\p_t$
\begin{equation}\label{omega}\Omega=\left\{P\in\mathcal{P}_t\left|\,\int_{S^d}|\nabla
P(x)|d\mu_d(x)<1\right.\right\}.
\end{equation}

Observe that if continuous maps $\x_i: \mathcal{P}_t\to S^d$,
$i=1,\ldots,N$, satisfy for all $P\in\partial\Omega$
$$\sum_{i=1}^N
P(\x_i(P))>0$$ then there exists a spherical $t$-design on $S^d$
consisting of $N$ points. To this end let us consider a map
$f:\p_t\to \p_t$ defined by
$$f(P)=
G_{\x_1(P)}+\cdots+G_{\x_N(P)}.$$ \rm Clearly
$$
(P,f(P))=\sum_{i=1}^N P(\x_i(P))
$$
for each $P\in\mathcal{P}_t$. Thus, applying Theorem B for the map
$f$, vector space $ \mathcal{P}_t$, and the subset $\Omega$ defined
in \eqref{omega} immediately gives us the existence of a polynomial
$\widetilde{P}\in\Omega$ such that $f(\widetilde{P})=0$. Hence,
by~\eqref{sph}, the images of this polynomial
$\x_1(\widetilde{P}),\ldots,\x_N(\widetilde{P})$ form a spherical
$t$-design in $S^d$ consisting of $N$ points. If additionally there
exists a constant $\lambda_d$ such that $|\x_i(P)-\x_j(P)|\geq
\lambda_d N^{-1/d}$ for all $i\neq j$, and $P\in\Omega$, then the
above mentioned spherical $t$-design is well separated, proving
Theorem~\ref{well sep}.

The maps $\x_i$, $i=1,\ldots,N$, will be constructed in Section
\ref{sep} below.

\section{Area-regular partitions and convex sets}
For $x,y\in S^d$ denote by $\dist(x,y)=\arccos((x,y))$ the geodesic
distance between $x$ and $y$. Also for a set $A\subset S^d$ define
the geodesic distance between $x$ and $A$ as follows
$$
\dist(x,A)=\inf_{y\in A}\dist(x,y).
$$
Recall that a spherical cap of radius $r$ with center at $x\in S^d$
is the set $$A(x,r)=\{z\in S^d|\,\mathrm{dist}(x,z)\leq r\}.$$

Below we will use extensively the notion of an area-regular
partition. Here is the definition.

Let $\mathcal{R}=\left\{R_1,\ldots, R_N\right\}$ be a finite
collection of closed sets $R_i\subset S^d$ such that $\cup_{i=1}^N
R_i=S^d$ and $\mu_d(R_i\cap R_j)=0$ for all $1\le i< j\le N$. The
partition $\mathcal{R}$ is called area-regular if
$\mu_d(R_i)=1/N$, $i=1,\ldots,N$. The
partition norm for $\mathcal{R}$ is defined by
$$ \|\mathcal{R}\|=\max_{R\in\mathcal{R}}\mathrm{diam\,}R=\max_{R\in\mathcal{R}}\max_{x,y\in R}\dist(x,y).$$
It is easy to prove using isodiametric inequality that each $R_i$
has diameter at least $c_dN^{-1/d}$. Therefore, $\|\mathcal{R}\|\ge
c_dN^{-1/d}$ for each $\mathcal{R}$. However for some $C_d$ and each
$N\in \N$ there exists an area-regular partition
$\mathcal{R}=\left\{R_1,\ldots, R_N\right\}$ of $S^d$ with diameter
at most $C_dN^{-1/d}$. Such area-regular partitions are used for many
optimization problems where a well distributed set of $N$ points on
a sphere having no concentration points is needed (see
e.g.~\cite{Alex},~\cite{BL},~\cite{MNW}, and~\cite{SK2}). In this
paper we need area-regular partitions of small diameter with
additional constraint of geodesic convexity. A subset $A\subset S^d$
is {\em geodesically convex} if any two points $x$, $y\in A$ can be
connected by a geodesic arc contained in $A$. The partition
$\mathcal{R}=\{R_1,\ldots,R_N\}$ is said to be {\em convex} if each
set $R_i,\;i=1,\ldots,N,$ is geodesically convex. First we will prove
the existence of convex area-regular partitions of ``small''
diameter.
\begin{proposition}\label{area-reg+} For each $N\in\N$
there exists a convex area-regular partition
$\mathcal{R}=\left\{R_1,\ldots, R_N\right\}$ such that
$\|\mathcal{R}\| \leq K_dN^{-1/d}$, where $K_d$ is a constant depending only
on~$d$.
\end{proposition}

The following construction for the sphere $S^2$ and $N=6n^2$, where
$n\in\N$, is given by Alexander in~\cite{Alex}. Let us first explain
his simple and elegant proof.

We begin with a spherical cube, and consider one of its facets. Using
$n - 1$ great circles from the pencil determined by two opposite
edges we can cut the facet into $n$ slices of equal area. Each slice
can be cut into $n$ quadrilaterals of equal area using great circles
in the pencil in the other pair of opposite edges of the face. The
diameters of the quadrilaterals are of the right magnitude.

This construction has an obvious generalization to higher
dimensions. Start with the appropriate spherical hypercube, then
divide each face into $n$ equal pieces, and so on. In this way we
obtain a convex partition of $S^d$ into $2(d+1)n^d$ parts of diameter at most $C_d/n$.

Now we generalize the approach of Alexander to prove
Proposition~\ref{area-reg+} for all $N\in \N$. For each $m\in\N$ and a
vector with positive coordinates $a=(a_1,\ldots,a_m)\in\R_+^m$
denote by $P(a)$ the $m$-dimensional rectangle
$[-a_1,a_1]\times\cdots\times[-a_m,a_m]$. Also, a measure
$d\eta(x)=\alpha(x)dx$ defined on $P(a)$ is said to be $M$-uniform
if
$$\alpha(x)\le M\alpha(y)\quad\text{for all }\, x,y\in P(a).$$ To prove
Proposition~\ref{area-reg+} we need the following lemma.
\begin{lemma}\label{rect}
Let $a\in\R_+^m$ with
$$
\max_{1\le i,j\le m}\frac{a_i}{a_j}\le B
$$
and $d\eta$ be a $M$-uniform measure defined on $P(a)$. Then for
each $N\in\N$ there exists a partition of $P(a)$ into $N$
$m$-rectangles $P_1,\ldots,P_N$ with facets parallel to the
corresponding facets of $P(a)$ such that
\begin{equation}
\label{area1} \int_{P_i}d\eta(x)=\frac 1N\int_{P(a)}d\eta(x),
\end{equation}
and
\begin{equation}
\label{area2} \mathrm{diam\,}{P_i}\le C(m,B,M)a_1N^{-1/m}
\end{equation}
for all $i=1,\ldots, N$.
\end{lemma}
\begin{proof}
We will prove the lemma by induction on $m$. For $m=1$, first we choose a point $t_1\in[-a_1,a_1]$ such that
$$
\int_{-a_1}^{t_1}d\eta(x)=\frac 1N\int_{-a_1}^{a_1}d\eta(x).
$$
Then similarly we choose $t_2\in(t_1,a_1]$ such that
$$
\int_{t_1}^{t_2}d\eta(x)=\frac 1N\int_{-a_1}^{a_1}d\eta(x)
$$
and so on. Finally, we get the partition of $P(a)=[-a_1,a_1]$ into $N$
segments $P_1=[-a_1,t_1]$, $P_2=[t_1,t_2]$,$\ldots$,
$P_N=[t_{N-1},a_1]$ satisfying~\eqref{area1} by its definition.
Moreover, $M$-uniformity of $\eta$ implies~\eqref{area2} with
$C(1,B,M)=2M$. Assume that the lemma is true for $m=l-1$. Let us
prove it for $m=l$. Put $k=[N^{1/l}]$, $s=[N/k]$, and $r= N-ks$.
Also for a fixed $a=(a_1,\ldots,a_l)\in\R_+^l$ denote by $P_1(a)$
the rectangle $[-a_1,a_1]\times\ldots\times[-a_{l-1},a_{l-1}]$. To
obtain a required partition of $P(a)$ first we choose step by step
points $-a_l=t_0<t_1<\ldots<t_{k-1}<t_{k}=a_l$ such that
\begin{equation*}
\int_{[t_{i-1},t_i]\times P_1(a)}d\eta(x)=\frac
{s+1}N\int_{P(a)}d\eta(x)
\end{equation*}
for all $i=1,\ldots,r$, and
\begin{equation*}
\int_{[t_{i-1},t_i]\times P_1(a)}d\eta(x)=\frac
{s}N\int_{P(a)}d\eta(x)
\end{equation*}
for all $i=r+1,\ldots,k$. Clearly,
\begin{equation}
\label{area2.45} |t_i-t_{i-1}|\le \frac{2(s+1)}NMa_l\le
4N^{-1/l}BMa_1.
\end{equation}
Consider the following measures on $P_1(a)$
$$
d\eta_i(x)=\alpha_i(x)dx,
$$
where
\begin{equation}
\label{area2.5}
\alpha_i(x)=\alpha_i(x_1,\ldots,x_{l-1})=\int_{t_{i-1}}^{t_i}\alpha(x_1,\ldots,x_{l})dx_l
\end{equation}
for all $(x_1,\ldots,x_{l-1})\in P_1(a)$ and $i=1,\ldots,k$. Clearly,
each $\eta_i$ is $2M^2$-uniform. Hence, by induction assumption for
each $i=1,\ldots,k$ and $N_i\in\N$ there exists a partition of $P_1$
into $N_i$ rectangles $P_{i,1},\ldots,P_{i,N_i}$ with facets
parallel to the corresponding facets of $P_1(a)$ such that
\begin{equation}
\label{area3} \int_{P_{i,j}}d\eta_i(x)=\frac
1{N_i}\int_{P_1(a)}d\eta_i(x),
\end{equation}
and
\begin{equation}
\label{area4} \mathrm{diam\, }{P_{i,j}}\le
C(l-1,B,M)a_1N_j^{-1/(l-1)}
\end{equation}
for all $j=1,\ldots,N_i$. Choose $N_i=s+1$ for $i=1,\ldots,r$ and
$N_i=s$ for $i=r+1,\ldots,k$. Consider the following partition of
$P(a)$ into $N$ rectangles
$$
P_{i,j}\times [t_{i-1},t_i],\quad i=1,\ldots,k,\,j=1,\ldots ,N_i.
$$
By~\eqref{area2.5} and~\eqref{area3} we immediately get that
$$
\int_{P_{i,j}\times
[t_{i-1},t_i]}d\eta(x)=\int_{P_{i,j}}d\eta_i(x)=\frac
1{N_i}\int_{P_1(a)}d\eta_i(x)=\frac1{N}\int_{P(a)}d\eta(x)
$$
for all $i=1,\ldots,k$ and $j=1,\ldots,N_i$. So, for this
partition~\eqref{area1} holds. Finally, combining~\eqref{area2.45}
with~\eqref{area4} we get~\eqref{area2} for some constant
$C(l,B,M)$. Lemma~\ref{rect} is proved.
\end{proof}
Now we are ready to construct the required convex area-regular
partitions.

\begin{proof}[Proof of Proposition~\ref{area-reg+}:]
We may assume that $N>8d^2$. First we consider the case when $N$ is even. For a $(d+1)$-rectangle $P(a)$, $a\in S^d$, denote
by
 $F_{2i-1}(a)$ its facet $x_i=a_i$ and by $F_{2i}(a)$ its facet
 $x_i=-a_i$, $i=1,\ldots,d+1$. One can naturally associate with $P(a)$ a convex partition
$\left\{R_1(a),\ldots, R_{2d+2}(a)\right\}$ of $S^d$, where
$R_i(a)=g(F_i(a))$, and $g(x)=x/|x|$ for all
$x\in\R^{d+1}\setminus\{0\}$.

Consider a one-parametric family of $(d+1)$-rectangles $P(a_\lambda)$,
where
$$
a_\lambda=\Big(\lambda,\sqrt{\frac{1-\lambda^2}d},\ldots,\sqrt{\frac{1-\lambda^2}d}\Big).
$$
Now we will choose such a $\lambda=\lambda(N)$ that our required
convex area-regular partition could be obtained as a subpartition of
$\left\{R_1(a_\lambda),\ldots, R_{2d+2}(a_\lambda)\right\}$.
Consider the function $G(\lambda)=\mu_d(R_1(a_\lambda))$. Clearly,
$G(1/\sqrt{d+1})=1/(2d+2)$ (in this case $P(a_\lambda)$ is a hypercube). On
the other hand $R_1(a_\lambda)$ is contained in the spherical cap $A((1,0,\ldots,0),\arccos\lambda)$. Therefore, we can estimate $G(\lambda)$ from above as
$$G(\lambda)\le\mu_d(A((1,0,\ldots,0),\arccos\lambda)).$$
Below we will use the following inequalities: for all $d\in\N$ and $N>8d^2$
$$
\mu_d(A((1,0,\ldots,0),\arccos(1-1/10d))\le\frac{N/2-d([N/(2d+2)]+1)}N\le\frac1{2d+2}.
$$
The left hand side inequality is very rough. We need this inequality with any constant strictly less than $1$ and depending only on $d$ in place of $1-1/10d$.
Now, by continuity of $G$ there
exists
\begin{equation} \label{lambda}
\lambda\in [1/\sqrt{d+1}, 1-1/10d] \end{equation}
such that
$$
G(\lambda)=\frac{N/2-d([N/(2d+2)]+1)}N.
$$
By symmetry arguments $\mu_d(R_1(a_\lambda))=\mu_d(R_2(a_\lambda))$
and
$$
\mu_d(R_3(a_\lambda))=\ldots=\mu_d(R_{2d+2}(a_\lambda))=\frac{[N/(2d+2)]+1}N.
$$
For each $i=1,\ldots, 2d+2$ consider the unique measure $\eta_i$ on $F_i(a_\lambda)$ such that
$\eta_i(E)=\mu_d(g(E))$ for each measurable set $E\subset F_i(a_\lambda)$ (this is indeed a measure, since $g$ is one-to-one).
Clearly,~\eqref{lambda} implies that each $\eta_i$ is $M_d$-uniform for large enough $M_d$.
Choose
$$
N_i=\frac{N/2-d([N/(2d+2)]+1)}N\quad\text{for } i=1,2$$ and
$$
N_i=\frac{[N/(2d+2)]+1}N\quad\text{for } i=3,\ldots,2d+2.
$$
Now applying Lemma~\ref{rect} for each $d$-rectangle
$F_i(a_\lambda)$ with measure $\eta_i$, $i=1,\ldots, 2d+2$ we can
get corresponding partition of $F_i(a_\lambda)$ into $N_i$ rectangles
$P_{i,j}$ such that
$$
\int_{P_{i,j}}d\eta_i(x)=\frac 1{N_i}\int_{F_i(a_\lambda)}d\eta_i(x)
$$
for $j=1,\ldots,N_i$. Moreover,
\begin{equation}
\label{area_diam}\mathrm{diam\,}{P_{i,j}}\le C_dN^{-1/d}.
\end{equation}
By its definition $\mu_d(g(P_{i,j}))=1/N$, $i=1,\ldots,2d+2$,
$j=1,\ldots, N_i$. Now we observe that each $g(P_{i,j})$ is a geodesically convex
closed set. Indeed, the image under the map $g$ of a line segment contained in $F_i$ is a geodesic arc on sphere $S^d$. Therefore, the image of the convex set $P_{i,j}$  is geodesically convex. Finally, the estimate
$\mathrm{diam\,}{g(P_{i,j})}\le C_dN^{-1/d}$ follows
from~\eqref{lambda} and~\eqref{area_diam}.

Now it remains to prove the proposition in the case when $N$ is odd. To this end we apply the same argument, with only difference that we replace $(d+1)$-rectangles $P(a_\lambda)$ by another family of polytopes $Q_{\lambda,\mu}$. Namely, for $\lambda,\mu\in(0,1)$ let $Q_{\lambda,\mu}$ be the convex hull of $2^{d+1}$ vertices $$\Bigg(\lambda,\pm\sqrt{\frac{1-\lambda^2}d},\ldots,\pm\sqrt{\frac{1-\lambda^2}d}\Bigg),
\quad\Bigg(-\mu,\pm\sqrt{\frac{1-\mu^2}d},\ldots,\pm\sqrt{\frac{1-\mu^2}d}\Bigg). $$
Consider the map $\phi:[-1,1]^{d+1}\to Q_{\lambda,\mu}$ given by
$$\phi(t_1,t):=\Bigg(t_1\cdot\frac{\lambda+\mu}{2}+
\frac{\lambda-\mu}{2},\bigg(t_1\cdot\frac{\sqrt{1-\lambda^2}+\sqrt{1-\mu^2}}{2\sqrt{d}}+
\frac{\sqrt{1-\lambda^2}-\sqrt{1-\mu^2}}{2\sqrt{d}}\bigg)t\Bigg), $$
where $t_1\in[-1,1]$ and $t\in[-1,1]^d$. Let $F_1,\ldots, F_{2d+2}$ be the facets of $[-1,1]^{d+1}$ numbered as before.
As in the case of even $N$ we can choose $\lambda\in [1/\sqrt{d+1}, 1-1/10d]$ such that
$$
\mu_d(g\circ\phi(F_1))=\frac{(N-1)/2-d([N/(2d+2)]+1)}N
$$
and then choose
$\mu\in [1/\sqrt{d+1}, 1-1/10d]$
such that
$$
\mu_d(g\circ\phi(F_2))=\frac{(N+1)/2-d([N/(2d+2)]+1)}N.
$$
Now by the symmetry argument for $i=3,\ldots,2d+2$
$$
\mu_d(g\circ\phi(F_i))=\frac{[N/(2d+2)]+1}N.
$$
Consider the pull-back measures $\eta_i$ on $F_i$ defined by $\eta_i(E):=\mu_d\big(g\circ\phi(E)\big)$ for any measurable subset $E\subset F_i$ (this is a well-defined measure, since $g\circ\phi$ is a.e. one-to-one). Clearly, each $\eta_i$ is $M_d$-uniform for large enough $M_d$, $i=1,\ldots,2d+2$. Applying again Lemma~\ref{rect} to the measures $\eta_i$  we get the corresponding area-regular partition of $S^d$. Also, the map $\phi$ has a useful property that the image of a hyperplane parallel to a facet of $[-1,1]^{d+1}$ is again a hyperplane. Therefore the partition is convex. Finally Lemma~\ref{rect} provides that the diameter of this partition is at most $K_dN^{-1/d}$ for some $K_d$ large enough.
\end{proof}

\noindent {\bf Remark.} The fact that $\mathcal{R}$ is convex easily
implies that each $R_i$, $i=1,\ldots,N$, contains a spherical cap of
radius $b_d N^{-1/d}$.\\  The following Theorem C states that an
arbitrary large enough and well distributed set of points is
``almost'' an equal weight quadrature formula in $S^d$;
see~\cite[Theorem 3.1]{MNW}.

\noindent {\sc {\bf Theorem C.}} {\em There exist constants $r_d>0$ and
$B_d>0$ such that for each  integer $m>B_d$, each $\eta\in(0,1)$, an
arbitrary convex area-regular partition
$\mathcal{R}=\{R_1,\ldots,R_N\}$ with
$\|\mathcal{R}\|<\eta\frac{r_d}m$, and each collection of points
$x_i\in R_i$, $i=1,\ldots,N$, the following inequalities
\begin{equation}
\label{Mhaskar}
(1-\eta)\int_{S^d}|P(x)|d\mu_d(x)\le\frac1N\sum_{i=1}^N|P(x_i)|\le(1+\eta)\int_{S^d}|P(x)|d\mu_d(x),
\end{equation}
hold for all polynomials $P$ of total degree at most $m$.}

To prove Theorem~\ref{well sep} we need the following lemma.

\begin{lemma}\label{nabla1}For each $\eta\in(0,1)$, an arbitrary convex area-regular partition $\mathcal{R}=\{R_1,\ldots,R_N\}$
with $\|\mathcal{R}\|<\eta\frac{r_{d}}{m+1}$, and any two collections of
points $x_i,\;y_i\in R_i$, $i=1,\ldots,N$, the following
inequalities
\begin{equation}
\label{Mhaskar-mod0}\frac 1N \sum_{i=1}^{N}|\,\nabla P(x_i)-\nabla
P(y_i)|\leq \, 8d\eta \,\int_{S^d} |\,\nabla P(x)|\,d\mu_d(x),
\end{equation}
\begin{equation}
\label{Mhaskar-mod} (1-8d\eta )\int_{S^d}|\,\nabla
P(x)|\,d\mu_d(x)<\frac1N\sum_{i=1}^N|\,\nabla
P(x_i)|<(1+8d\eta)\int_{S^d}|\,\nabla P(x)|\,d\mu_d(x),
\end{equation}
hold for all polynomials $P$ of total degree $m\ge B_d$.
The constants $r_d$ and $B_d$ are given by Theorem C.
\end{lemma}
\begin{proof}
First we will prove~\eqref{Mhaskar-mod0}. Since $|\nabla
P|=\sqrt{P_1^2+\ldots+P_{d+1}^2}$, where $P_i\in \mathcal{P}_{m+1}$
for $i=1,\ldots,d+1$,  we have
\begin{equation}
\label{eq0} \frac 1N \sum_{i=1}^{N}|\nabla P(x_i)-\nabla
P(y_i)|\le\frac 1N
\sum_{i=1}^{N}\sum_{j=1}^{d+1}|P_j(x_i)-P_j(y_i)|.
\end{equation}
Now, for each $Q\in\mathcal{P}_{m+1}$ we will estimate the value
$$
\frac 1N \sum_{i=1}^{N}|Q(x_i)-Q(y_i)|.
$$
Let $I_1$ be the set of indexes $i=1,\ldots,N$ such that the value
$Q(x)$ has the same sign for all $x\in R_i$, and $I_2$ be the set of
all other indexes, that is the set of $i=1,\ldots,N$, for which
there exists a point $x\in R_i$ with $Q(x)=0$. Let $I_3$ be the set
of indexes $i=1,\ldots,N$ such that $|Q(x_i)|\ge |Q(y_i)|$. Put
$z_i:=x_i$, if $i\in I_3$ and $z_i:=y_i$ otherwise. Put $t_i:=y_i$,
if $i\in I_1\cap I_3$, and $t_i:=x_i$, if $i\in I_1\setminus I_3$.
For $i\in I_2$, let $t_i$ be a point in $R_i$ such that $Q(t_i)=0$.
We have
$$
\frac 1N \sum_{i=1}^{N}|Q(x_i)-Q(y_i)|=\frac 1N \sum_{i\in
I_1}|Q(z_i)|-\frac 1N\sum_{i\in I_1}|Q(t_i)|+\frac 1N\sum_{i\in
I_2}|Q(x_i)-Q(y_i)|
$$
$$
\le\frac 1N \sum_{i\in I_1}|Q(z_i)|-\frac 1N\sum_{i\in I_1}|Q(t_i)|+
\frac 2N\sum_{i\in I_2}|Q(z_i)|-\frac 2N\sum_{i\in I_2}|Q(t_i)|
$$
$$
\le\frac 2N\sum_{i=1}^N|Q(z_i)|-\frac 2N\sum_{i=1}^N|Q(t_i)|,
$$
where $z_i$, $t_i\in R_i$, $i=1,\ldots,N$. Thus, by~\eqref{Mhaskar}
we have
$$
\frac 1N
\sum_{i=1}^{N}|Q(x_i)-Q(y_i)|\le 4\eta\int_{S^d}|Q(x)|d\mu_d(x).
$$
So, the inequality~\eqref{eq0} implies
$$
\frac 1N \sum_{i=1}^{N}|\nabla P(x_i)-\nabla P(y_i)|\le 4\eta\
\sum_{j=1}^{d+1}\int_{S^d}|\nabla P_j(x)|d\mu_d(x)\le
8d\eta\int_{S^d}|\nabla P(x)|d\mu_d(x).
$$
This proves~\eqref{Mhaskar-mod0}. Now by the mean value theorem
there exist $y_i\in R_i$ such that
$$\frac 1N|\nabla
P(y_i)|=\int_{R_i}|\nabla P(x)|d\mu_d(x),\quad i=1,\ldots,N.
$$
Finally, we obtain the inequality~\eqref{Mhaskar-mod}
from~\eqref{Mhaskar-mod0} and the following easy inequalities
$$
\frac1N\sum_{i=1}^N|\nabla P(y_i)|-\frac1N\sum_{i=1}^N|\nabla
P(x_i)-\nabla P(y_i)|\le\frac1N\sum_{i=1}^N|\nabla P(x_i)|
$$
$$
\le\frac1N\sum_{i=1}^N|\nabla P(y_i)|+\frac1N\sum_{i=1}^N|\nabla
P(x_i)-\nabla P(y_i)|.
$$
\end{proof}

The following lemma is crucial to construct the maps
$\x_1,\ldots,\x_N: \mathcal{P}_t\to S^d$ in the next section.

\begin{lemma}\label{conv}For  $x\in S^d$ denote by $T_{x}$ the space of all vectors $y\in \R^{d+1}$ with $(x,y)=0$.
 Let $R\subset S^d$ be a closed geodesically convex set with $\mathrm{diam\,} R<\pi/2$.
 Then for each interior point $x\in R$ and $y\in T_{x}\setminus \{0\}$ the following holds:\\
(i) there exists a unique $x_{\mathrm{max}}\in R$ with $(x_{\mathrm{max}},y)=\max_{z\in R} (z,y)$;\\
(ii) the map $M_{x}:T_{x}\setminus \{0\}\to R$ given by $y\to x_{\mathrm{max}}$ is continuous on $T_{x}\setminus \{0\}$;\\
(iii) for each $w\in R$ and a geodesic $\gamma:[0,1]\to R$ with
$\gamma(0)=x_{\mathrm{max}}$, $\gamma(1)=w$ the function
$(y,\gamma(h))$ is decreasing on $[0,1]$.
\end{lemma}
\begin{proof}
Consider an orthogonal projection $p:\R^{d+1}\to T_x$ given by
$$p(z)=z-(x,z)x.$$
Clearly,
\begin{equation}
\label{S} (z,y)=(p(z),y) \end{equation} for all $z\in R$.

Denote by $S=p(R)$ the image of $R$ under the projection $p$. Since
$\dist(x,z)<\pi/2$ for each $z\in R$, then $p$ is a homeomorphism between
$R$ and $S$ and the inverse map is given by\rm
$$
p^{-1}(u)=u+\sqrt{1-|u|^2}x,\quad u\in S.
$$
Now we will show that $S$ is a strictly convex subset of $T_x$, i.e.
for each pair of distinct points $u$, $v\in S$ and each $h\in(0,1)$ the point
$hu+(1-h)v$ is an interior point of $S$. To this end we note that
\begin{equation}
\label{int1} p^{-1}(hu+(1-h)v)=hp^{-1}(u)+(1-h)p^{-1}(v)+\alpha x,
\end{equation}
where $\alpha>0$. We will use the following simple statement:\\
{\it If $w_1$, $w_2\in S^d$ are such that $(w_1,w_2)>0$, and
$w_3=\alpha_1w_1+\alpha_2w_2\in S^d$ for some $\alpha_1,
\alpha_2>0$, then $w_3$ lies on the shortest geodesic connecting
$w_1$ and $w_2$.}\\
This statement and the fact that $p^{-1}(u),\;p^{-1}(v)\in R$
immediately imply that $z/|z|\in R$, where
$z=hp^{-1}(u)+(1-h)p^{-1}(v)$. Hence, applying again the statement
for $z/|z|$ and $x$ we get by~\eqref{int1} that
$p^{-1}(hu+(1-h)v)\in R$. Moreover, since $x$ is an interior point
of $R$, and $\alpha>0$, then $p^{-1}(hu+(1-h)v)$ is an interior
point of $R$ as well, and therefore $hu+(1-h)v$ is an interior point
of $S$.

To prove (i) we will use the known fact that a nonconstant linear
function given on a closed strictly convex subset in $\R^{d}$
attains its maximum in a unique point. Using this fact we get that
there exists a unique $z_{\mathrm{max}}\in S$ such that
$(z_{\mathrm{max}},y)=\max_{z\in S}(z,y)$. Finally, by~\eqref{S} we
get that $x_{\mathrm{max}}=p^{-1}(z_{\mathrm{max}})$.

Now we will prove (ii). Since $p$ is a homeomorphism it suffices to
show that the composition map $p\,\circ M_x:T_{x}\setminus \{0\}\to
S$ is continuous. Note that $(y,p\,\circ M_x(y))=\max_{z\in S}
(y,z)$. Since $S$ is a closed strictly convex set then for each
$\epsilon>0$ there exists $\delta=\delta(\epsilon)>0$ such that for
all $v\in T_x$ with $|v-y|<\delta$ the diameter of the set $\{z\in
S|(v,z)>(v,p(M_x(y)))\}$ is less than $\epsilon$. Hence,
$|p(M_x(y))-p(M_x(v))|<\epsilon$. Thus, the map $p\circ M_x$ is
continuous at $y$, and so is $M_x$. This proves (ii).

Finally, we prove part (iii) of the lemma. Let $G$ be the great circle
containing $x_\mathrm{max}$ and $w$. There is a unique point
$w_\mathrm{max}\in G$ such that $(y,w_{\mathrm{max}})=\max_{z\in\,
G} (y,z)$. Now for each $z\in G$ we have
\begin{equation}
\label{G} (y,z)=(y,w_\mathrm{max})(z,w_{\mathrm{max}}).
\end{equation}
Hence the scalar product $(y,z)$ is increasing on both geodesic
arcs connecting $-w_\mathrm{max}$ and $w_\mathrm{max}$. The geodesic
$\gamma: [0,1]\to R$ is an arc of $G$. To prove (iii) it is enough
to show that both $-w_\mathrm{max}$ and $w_\mathrm{max}$ are outside
of the arc $\gamma$. The point $w_\mathrm{max}$ is outside of
$\gamma$ by the definition of $x_\mathrm{max}$. Moreover,
$(y,x_{\mathrm{max}})>(y,x)=0$. Therefore, substituting
$z=x_\mathrm{max}$ to~\eqref{G}  we see that
$(x_\mathrm{max},w_\mathrm{max})>0$\rm. Hence,
$$\mathrm{dist}(x_{\mathrm{max}},-w_{\mathrm{max}})>\pi/2.$$ Finally,
the fact that $\mathrm{diam\,} R<\pi/2$ implies that
$-w_{\mathrm{max}}$ is outside of $\gamma$ as well. Thus, the
function $(y,\gamma(h))$ is decreasing on $[0,1]$.  \end{proof}

\section{Proof of Theorem 1}\label{sep}

Fix $t\in\N$. In Section 2 we explained that it is enough to construct an
$N$-tuple of continuous maps $\x_1,\ldots,\x_N: \mathcal{P}_t\to
S^d$ such that
\begin{equation*}
\frac1N\sum_{i=1}^N P(\x_i(P))>0
\end{equation*}
for all $P\in\partial\Omega$ and
\begin{equation*}\mathrm{dist}(\x_i(P),\x_j(P))>\lambda_d
N^{-1/d},\quad 1\le i< j\le N,\end{equation*} for all $P\in\Omega$,
where $\Omega$ is given by~\eqref{omega}.

Fix $\epsilon,\delta,\eta>0$. Consider the function
$$g_\epsilon(t):=\begin{cases}t/\epsilon & \text{if $t\leq \epsilon$,}\\ 1& \mbox{otherwise.}\end{cases} $$
Let $N>t^d(\frac{2\,
K_d}{\eta\, r_d})^d$ and $\mathcal{R}=\{R_1,\ldots,R_N\} $ be an area-regular partition
provided by Proposition~\ref{area-reg+}. For each $i=1,\ldots,N$ choose a point $x_i\in
R_i$ such that $R_i$ contains a spherical cap of radius  $b_d
N^{-1/d}$ with center at $x_i$. Recall that $\|R\|\le K_dN^{-1/d}$.

Let $P\in\p_t$. By Lemma \ref{conv} for each $i=1,\ldots,N$  there
exists a unique $z_i=z_i(P)\in R_i$ satisfying
\begin{equation}\label{z_i}(z_i,\nabla P(x_i))=\max_{x\in R_i}(x,\nabla P(x_i)), \end{equation}
 provided that $\nabla P(x_i)\neq 0$. In the case $\nabla P(x_i)=0$ put $z_i=x_i$. Let $\gamma_{[x_i,z_i]}:[0,1]\to R_i$
  be a geodesic connecting $x_i$ and $z_i$.  We assume that the curve $\gamma_{[x_i,z_i]}$ has an equal-speed parametrization, i.e. the derivative with respect to parameter $h$ satisfies $|\gamma_{[x_i,z_i]}'(h)|=\mathrm{dist}(x_i,z_i)$ for $h\in(0,1)$. Define
\begin{equation}\label{map2}\x_i(P):=\gamma_{[x_i,z_i]}\left((1-\delta)\,g_\epsilon(\,|\,\nabla P(x_i)|\,)\right).\end{equation}
By the definition of $g_\epsilon$ the map $\x_i:\mathcal{P}_t\to S^d$ is continuous in a small neighborhood of the set $\{P\in\p_t| \nabla P(x_i)=0\}$. On the other
hand, part (ii) of Lemma~\ref{conv} implies that $\x_i$ is continuous on the set $\{P\in\p_t: \nabla P(x_i)\neq0\}$. Thus the maps $\x_1,\ldots \x_N$ are
continuous in $\p_t$. The following Lemma \ref{M_x} will finish the
proof of Theorem~\ref{well sep}.

\begin{lemma}\label{M_x} There exist constants $\epsilon,\delta,\eta$ depending only on $d$
such that for each $N>t^d(\frac{2\, K_d}{\eta\, r_d})^d$ the $N$-tuple of
maps $\x_1,\ldots,\x_N: \mathcal{P}_{t}\to S^d$  defined
by~\eqref{map2} satisfies the following properties:\\
\begin{equation}\label{1}
\frac1N\sum_{i=1}^N P(\x_i(P))>0
\end{equation}
for all $P\in\partial\Omega$ and
\begin{equation}\label{2}\mathrm{dist}(\x_i(P),\x_j(P))>\lambda_d
N^{-1/d},\quad 1\le  i< j\le N,\end{equation} for all $P\in\Omega$.
\end{lemma}

\begin{proof} Fix $P\in\mathcal{P}_t$.
For each $i=1,\ldots,N$ choose $z_{i,\mathrm{max}}\in R_i$ such that
$P(z_{i,\mathrm{max}})=\max_{x\in R_i} P(x)$. Denote
$y_{i,\epsilon}:=\gamma_{[x_i,z_i]}(g_\epsilon(|\nabla P(x_i)|)),$
where $z_i$ and $\gamma$ are as in~\eqref{map2}.  We can split the
sum \eqref{1} into four pieces
\begin{align}\label{split}\frac 1N \sum_{i=1}^N P(\x_i(P))&=\frac 1N \sum_{i=1}^N P(z_{i,\mathrm{max}})\\ &+
\frac 1N \sum_{i=1}^N( P(z_i)- P(z_{i,\mathrm{max}}))\notag\\
&+\frac 1N \sum_{i=1}^N( P(y_{i,\epsilon})- P(z_{i}))
\notag\\&+\frac 1N \sum_{i=1}^N( P(\x_i(P))-
P(y_{i,\epsilon}))\notag.\end{align} We will estimate each of these
sums separately.

Clearly,
\begin{equation}\label{sum P(zi,max)}\frac 1N \sum_{i=1}^N P(z_{i,\mathrm{max}})= \sum_{i=1}^N
\int_{R_i}\big(P(z_{i,\mathrm{max}})-P(x)\big)d\mu_d(x).\end{equation} Now
note that if $z_{i,\mathrm{max}}\not\in\partial R_i$ for some
$i=1,\ldots,N$ then $\nabla P(z_{i,\mathrm{max}})=0$, therefore
\begin{equation}\label{P(zi,max)}P(z_{i,\mathrm{max}})-P(x)\geq \min_{y\in R_i} |\,\nabla P(y)|\,\mathrm{dist}(x,\partial R_i)\end{equation}
for all $i=1,\ldots,N$ and  $x\in R_i$. Let $A_i$ be a spherical cap
of radius $b_dN^{-1/d}$ contained in $R_i$. Since $\|R\|\le
K_dN^{-1/d}$ we obtain that
\begin{equation}\label{beta}\int_{R_i}\mathrm{dist}(x,\partial
R_i)d\mu_d(x)\geq \int_{A_i}\mathrm{dist}(x,\partial A_i)
d\mu_d(x)\geq\beta_d\frac{\|\mathcal{R}\|}{N},
\end{equation}
for some constant
$\beta_d$. Thus, it follows from \eqref{sum P(zi,max)} and
\eqref{P(zi,max)} that
$$\frac 1N \sum_{i=1}^N P(z_{i,\mathrm{max}})\geq \beta_d\frac {\|\mathcal{R}\|}{N}\sum_{i=1}^N\min_{y\in R_i}
|\,\nabla P(y)|.$$ Since by Theorem~\ref{area-reg+} we have
$\|\mathcal{R}\|<\eta\, r_d/(t+1)$, using Lemma \ref{nabla1} we
arrive at
\begin{equation}\label{split1}\frac 1N \sum_{i=1}^N P(z_{i,\mathrm{max}})\geq \|\mathcal{R}\| \,\beta_d\,(1-8d\eta)\int_{S^d}|\,\nabla P(x)|\,d\mu_d(x).
\end{equation}

Next we estimate the sum
$$\frac 1N \sum_{i=1}^N( P(z_{i,\mathrm{max}})-P(z_i) ).$$
Let $\gamma:[0,1]\to R_i$ be a geodesic connecting $z_i$ and $z_{i,\mathrm{max}}$. We can write
$$P(z_{i,\mathrm{max}})-P(z_i)=\int_{0}^{1} (\nabla P(\gamma(h)),\gamma'(h))\,dh\geq 0.$$
By Lemma \ref{conv} (iii) the inequality $(\gamma'(h),\nabla
P(x_i))<0$ holds for all $h\in (0,1)$. Thus, we have
$$\int_{0}^{1} (\nabla P(\gamma(h)),\gamma'(h))dh\leq
\int_{0}^{1} (\nabla P(\gamma(h))-\nabla P(x_i),\gamma'(h))\,dh
$$
$$\leq\int_{0}^{1} |\,\nabla P(\gamma(h))-\nabla P(x_i)|\,|\,\gamma'(h)|\,dh\leq
\mathrm{diam}\, R_i\,\max_{x\in R_i} |\,\nabla P(x)-\nabla P(x_i)|.
$$
Using Lemma \ref{nabla1} we arrive at
\begin{equation}\label{split2}\frac 1N \sum_{i=1}^N|\,P(z_{i,\mathrm{max}})-P(z_i)|\leq \|\mathcal{R}\| \,8d\eta\,\int_{S^d}|\,\nabla P(x)|\,d\mu_d(x). \end{equation}

Now we estimate the third sum in the left-hand side of
\eqref{split}.
Recall that $$y_{i,\epsilon}=\begin{cases}z_i&\mbox{for}\; |\nabla P(x_i)|\geq\epsilon\\
 \gamma_{[x_i,z_i]}(|\nabla P(x_i)|/\epsilon)&\mbox{otherwise}.\end{cases}$$ Hence, we obtain
 $$\frac 1N \sum_{i=1}^N( P(y_{i,\epsilon})- P(z_{i}))=\frac 1N\sum_{i:|\nabla P(x_i)|<\epsilon}( P(y_{i,\epsilon})- P(z_{i})).$$
Since $z_i$ and $y_{i,\epsilon}$ are both in $R_i$, we can write an obvious estimate
$$ |P(y_{i,\epsilon})- P(z_{i})|\leq \mathrm{diam}\, R_i\max_{x\in R_i} |\,\nabla P(x)|.$$
For each $i=1,\ldots,N$ choose $w_i\in R_i$ such that $|\,\nabla P(w_i)|=\max_{x\in R_i} |\,\nabla P(x)|$. Then
$$\frac 1N\sum_{i:|\nabla P(x_i)|<\epsilon}|\,\nabla  P(w_i)|\leq $$ $$\frac 1N\sum_{i:|\nabla  P(x_i)|<\epsilon}\big(|\,\nabla P(w_i)|-|\,\nabla  P(x_i)|\,\big)+\epsilon+\frac 1N\sum_{i:|\nabla P(x_i)|\geq\epsilon}\big(|\,\nabla P(x_i)|-|\,\nabla P(x_i)|\,\big).$$
Thus, Lemma~\ref{nabla1} implies that
$$\frac 1N\sum_{i:|\nabla P(x_i)|<\epsilon}|\,\nabla  P(w_i)|\leq\epsilon+8d\eta\int_{S^d}|\,\nabla P(x)|\,d\mu_d(x). $$ Hence, we arrive at
\begin{equation}\label{split3}\left|\frac 1N \sum_{i=1}^N\big( P(y_{i,\epsilon})- P(z_{i})\big)\right|\leq \|\mathcal{R}\|\,\Big(\epsilon + 8d\eta\,\int_{S^d} |\,\nabla P(x)|\,d\mu_d(x)\Big).\end{equation}

It remains to estimate the sum
$$\frac 1N \sum_{i=1}^N( P(\x_{i}(P))-P(y_{i,\epsilon}) ).$$
The distance between $\x_i(P)$ and $y_{i,\epsilon}$ is less than $\delta\,\|\mathcal{R}\|$. Hence,
\begin{equation*}\left|\frac 1N \sum_{i=1}^N( P(\x_{i}(P))-P(y_{i,\epsilon}) )\right|\leq\frac {\delta\,\|\mathcal{R}\|}{N} \sum_{i=1}^N\max_{x\in R_i} |\,\nabla P(x)|. \end{equation*}
Using again Lemma \ref{nabla1} we arrive at

\begin{equation}\label{split4}\left|\frac 1N \sum_{i=1}^N( P(\x_{i})-P(y_{i,\epsilon})
)\right|\leq\|\mathcal{R}\|\,\delta\,(1+8d\eta)\int_{S^d} |\,\nabla P(x)|\,d\mu_d(x). \end{equation}

Now for $P\in\partial\Omega$, we get by \eqref{split}, \eqref{split1}, \eqref{split2},
\eqref{split3}, and \eqref{split4} that
\begin{equation}\label{last}\frac 1N \sum_{i=1}^N P(\x_i(P))\geq \|\mathcal{R}\|(
\beta_d(1-8d\eta)-8d\eta-(8d\eta+\epsilon)-\delta(1+8d\eta)).\end{equation}

Take $\eta=\beta_d/(48 d), \delta=\beta_d/3$ and
$\epsilon=\beta_d/12$, where $\beta_d$ is provided by~\eqref{beta}.
Without loss of generality we may assume that $\beta_d<1$. Thus we
get
$$ \beta_d(1-8d\eta)-8d\eta-(8d\eta+\epsilon)-\delta(1+8d\eta)>0,$$
which together with~\eqref{last} imply~\eqref{1}.

It remains to show the separation property~\eqref{2}. Fix $P\in
\Omega$ and $i\in\overline{1,\ldots,N}$. By the
definition~\eqref{map2}, $\x_i(P)$ is in $R_i$. Thus to
prove~\eqref{2} it is enough to show that $\dist(\x_i(P),\partial
R_i)\ge\lambda_dN^{-1/d}$ for some constant $\lambda_d$. Recall that
$\x_i(0)=x_i$, and $R_i$ contains a spherical cap of radius $r=b_d
N^{-1/d}$ with center at $x_i$. The main reason why $x_i(P)$ is
``far away'' from the boundary $\partial R_i$ is because $\x_i(P)$
lies on the geodesic $\gamma_{[x_i,z_i]}$, where $z_i\in R_i$, and
$$\dist(x_i,\x_i(P))\leq(1-\delta)\dist(x_i,z_i).$$
We will also use the fact that $R_i$ is geodesically convex and
contains a spherical cap of ``big'' radius with center at $x_i$.

Denote by $T_{x_i}$ the space of all vectors in $\R^{d+1}$
orthogonal to $x_i$ and let $p:\R^{d+1}\to T_{x_i}$ be the
orthogonal projection $$p(z)=z-(x_i,z)x_i.$$ As we have pointed out
in Lemma~\ref{conv} the image $S_i=p(R_i)$ is a convex subset in
$T_{x_i}$. Clearly, $$\dist(\x_i(P),\partial R_i)\ge
\dist_{\mathrm{euc}}(p(\x_i(P)),\partial S_i),$$ where $\dist_{\mathrm{euc}}(z,\partial S_i)$ stands for
the Euclidean distance between point $z$ and the set $\partial S_i$
in $T_{x_i}$. The point $p(\x_i(P))$ lies between the points
$p(x_i)=0$ and $p(z_i)$ on the line connecting them. Thus we have
\begin{equation}\label{sep1}
|p(x_i)-p(z_i)|=\sin(\dist(x_i,z_i)),\end{equation} and
\begin{equation}\label{sep2}|p(x_i)-p(\x_i(P))|\leq\sin((1-\delta)\dist(x_i,z_i)).
\end{equation}
Moreover, the fact that $R_i$ contains the spherical cap $A(x_i,r)$
implies that
\begin{equation}\label{sep3}\dist_{\mathrm{euc}}(p(x_i),\partial S_i)\ge \sin r.
\end{equation}
Finally we note that the function $\dist_{\mathrm{euc}}(z,\partial S_i)$ is concave on
$S_i$. Therefore by~\eqref{sep1}-\eqref{sep3} we get
$$
\dist(\x_i(P),\partial R_i)\ge \frac{\delta}2\sin
r\ge\lambda_dN^{-1/d},
$$
which implies~\eqref{2}. Lemma \ref{M_x} is proved.

\end{proof}
\noindent
{\bf Acknowledgements.}
The authors thank the Mathematisches Forschungsinstitut Oberwolfach for their
hospitality during the preparation of this manuscript and for providing a stimulating atmosphere for research.
This paper is partially supported by the Centre for Advanced Study at the Norwegian Academy of Science and Letters in Oslo.

{\footnotesize \noindent Department of Mathematical Analysis, Taras
Shevchenko National University of Kyiv,  Volodymyrska~64, 01033~Kyiv, Ukraine\\
and\\
Department of Mathematical Sciences, Norwegian University of Science and Technology, NO-7491 Trondheim, Norway\\
{\it Email address: andriybond@gmail.com}\\
\vspace{0.2cm}

\noindent
Max Planck Institute for Mathematics, Vivatsgasse 7, 53111 Bonn, Germany\\and\\
Department of Mathematical Analysis, Taras
Shevchenko National University of Kyiv,  Volodymyrska~64, 01033~Kyiv, Ukraine\\
{\it Email address: danrad@mpim-bonn.mpg.de}\\
\vspace{0.2cm}

\noindent
Max Planck Institute for Mathematics, Vivatsgasse 7, 53111 Bonn, Germany\\
{\it Email address: viazovska@gmail.com}}

\end{document}